\documentclass[12pt]{amsart}

\usepackage{amssymb, amsfonts, amsmath, amscd, amsthm}
\usepackage{hyperref, eucal, enumerate, mathrsfs}
\usepackage[all]{xy}

\newtheorem{lemma}[subsection]{Lemma}
\newtheorem{proposition}[subsection]{Proposition}
\newtheorem{theorem}[subsection]{Theorem}
\newtheorem{corollary}[subsection]{Corollary}

\theoremstyle{definition}
\newtheorem{pg}[subsection]{}
\newtheorem{definition}[subsection]{Definition}
\newtheorem{remark}[subsection]{Remark}

\numberwithin{equation}{subsection}

\DeclareMathOperator{\acn}{acn}
\DeclareMathOperator{\APerf}{APerf}
\DeclareMathOperator{\bbE}{\mathbb{E}}
\DeclareMathOperator{\CAlg}{CAlg}
\DeclareMathOperator{\Cat}{\mathcal{C}at}
\newcommand{\Cech}{\v{C}ech\,}
\DeclareMathOperator{\cn}{cn}
\DeclareMathOperator{\colim}{colim} 
\DeclareMathOperator{\ex}{ex}
\DeclareMathOperator{\Ext}{Ext} 
\DeclareMathOperator{\Fun}{Fun}
\DeclareMathOperator{\LMod}{LMod}
\DeclareMathOperator{\Map}{Map}
\DeclareMathOperator{\Mod}{Mod}
\DeclareMathOperator{\op}{op}
\DeclareMathOperator{\Perf}{Perf}
\DeclareMathOperator{\QCoh}{QCoh}
\DeclareMathOperator{\Spec}{Spec}
\DeclareMathOperator{\SSet}{\mathcal{S}}
\DeclareMathOperator{\Sym}{Sym} 
\DeclareMathOperator{\Tot}{Tot}

\DeclareMathOperator{\calC}{\mathcal{C}}
\DeclareMathOperator{\calD}{\mathcal{D}}

\DeclareMathOperator{\sfX}{\mathsf{X}}

\newcommand{\Pull}[8]{\xymatrix{#1\ar[r]^-{#5}\ar[d]^-{#6} & #2 \ar[d]^-{#7} \\ #3 \ar[r]^-{#8} & #4}}

\title{Spectral Excision and Descent for Almost Perfect Complexes}
\author{Chang-Yeon Chough}
\address{Department of Mathematics, Sogang University, Seoul 04107, Republic of Korea}
\email{chough@sogang.ac.kr}


\begin{document}

\begin{abstract}
	We show that almost perfect complexes of commutative ring spectra satisfy excision and $v$-descent. These results generalize Milnor excision for perfect complexes of ordinary commutative rings and $v$-descent for almost perfect complexes of locally noetherian derived stacks by Halpern-Leistner and Preygel, respectively. 
\end{abstract}

\setcounter{tocdepth}{1} 
\maketitle
\tableofcontents

\section{Introduction}

\begin{pg}\label{perfect complexes in the case of affines}
	Suppose we are given a square of associative rings $\sigma$:
$$
\Pull{A}{A'}{B}{B'}{}{}{}{}
$$
for which the map $B\rightarrow B'$ is surjective. According to Milnor (see \cite[\textsection 2]{MR0349811}), the image of $\sigma$ under the functor which assigns to each associative ring $R$ the category of finitely generated projective $R$-modules is a pullback diagram of categories. To work in the more general context of structured ring spectra, recall that if $R$ is a commutative ring, then a chain complex of $R$-modules $M$ (viewed as an object of the derived category of $R$-modules) is \emph{perfect} if it is quasi-isomorphic to a bounded complex of finitely generated projective $R$-modules (see, for example, \cite[\href{https://stacks.math.columbia.edu/tag/0657}{Tag 0657}]{stacks-project}). More generally, we say that $M$ is \emph{pseudo-coherent} it if is quasi-isomorphic to a bounded above complex of finitely generated free $R$-modules (see, for example, \cite[p.79]{MR354655}). Suppose now that $R$ is an $\bbE_\infty$-ring in the sense of \cite[7.1.0.1]{lurie2017ha} and let $\Mod_R$ denote the $\infty$-category of $R$-modules (see \cite[7.1.1.2]{lurie2017ha}). Then the notions of perfect and pseudo-coherent modules over commutative rings can be generalized to the setting of $\bbE_\infty$-rings, and we obtain the notions of \emph{perfect} and \emph{almost perfect} modules over $R$, respectively (see \cite[7.2.4.1]{lurie2017ha} and \cite[7.2.4.10]{lurie2017ha}). We will denote by $\Perf(R)$ and $\APerf(R)$ the full subcategories of $\Mod_R$ spanned by the perfect and almost perfect $R$-modules, respectively. The constructions $R \mapsto \Perf(R), \APerf(R)$ determine functors $\Perf, \APerf: \CAlg^{\cn} \rightarrow \Cat_\infty$, where $\CAlg^{\cn}$ and $\Cat_\infty$ denote the $\infty$-category of connective $\bbE_\infty$-rings and the $\infty$-category of $\infty$-categories, respectively (see \cite[p.1201]{lurie2017ha} and \cite[3.0.0.1]{MR2522659}).
\end{pg}

\begin{pg}
	One of the main results in this paper is the following analogue of Milnor's result in the setting of $\bbE_\infty$-rings:
\end{pg}

\begin{theorem}\label{excision of commutative ring spectra} 
	Suppose we are given a pullback square of connective $\bbE_\infty$-rings $\sigma:$
$$
\Pull{A}{A'}{B}{B'.}{}{}{}{}
$$
If the induced map $\pi_0B \rightarrow \pi_0B'$ is a surjection of commutative rings, then the diagram of $\infty$-categories $\APerf(\sigma):$
$$
\Pull{\APerf(A)}{\APerf(A')}{\APerf(B)}{\APerf(B')}{}{}{}{}
$$
determined by the extension of scalars functors is a pullback square in the $\infty$-category $\Cat_\infty$.
\end{theorem}

\begin{remark}
	In our proof of \ref{excision of commutative ring spectra}, we will make use of the notion of a \emph{universal descent morphism} of $\bbE_\infty$-rings in the sense of \cite[D.3.1.1]{lurie2018sag}, which was introduced originally by Akhil Mathew in \cite[3.18]{MR3459022}. We note that the notion of universal descent morphisms makes sense more generally for $\bbE_2$-rings. However, it is in the commutative setting that the class of universal descent morphisms has the descent property of \cite[D.3.5.8]{lurie2018sag}, which will play an important role in our discussion of almost perfect complexes. For this reason, we focus our attention to the case of $\bbE_\infty$-rings (unlike Milnor's excision, which works for associative rings).
\end{remark}

\begin{remark}\label{excision of commutative ring spectra under the stronger assumption}
	Under the additional assumption that the underlying map of commutative rings $\pi_0A' \rightarrow \pi_0B'$ is surjective, $\APerf(\sigma)$ is already known to be a pullback diagram by virtue of \cite[16.2.0.2]{lurie2018sag} and \cite[16.2.3.1]{lurie2018sag}. 
\end{remark}

\begin{remark}
	In the special case where $\sigma$ is a pullback square of connective $\bbE_\infty$-rings for which the map $\pi_0B \rightarrow \pi_0B'$ is a surjection whose kernel is a nilpotent ideal of $\pi_0B$, \ref{excision of commutative ring spectra} can be deduced from \cite[16.2.0.2]{lurie2018sag} and \cite[2.7.3.2]{lurie2018sag}. 
\end{remark}

\begin{remark}\label{excision of commutative ring spectra for perfect complexes}
	By restricting the equivalence of \ref{excision of commutative ring spectra} to the full subcategories spanned by the dualizable objects (see the proof of \ref{universal descent implies universal perf-descent}), we immediately deduce that the canonical map $\Perf(A) \rightarrow \Perf(A')\times_{\Perf(B')}\Perf(B)$ is an equivalence of $\infty$-categories. We remark that this equivalence is proven in \cite[2.23]{MR3459022} under the additional assumption that the map $\pi_0A' \rightarrow \pi_0B'$ is a surjection. 
\end{remark}

\begin{remark}
	According to \cite[1.17]{MR4024564} of Markus Land and Georg Tamme, $\Perf(\sigma)$ is a pullback diagram of $\infty$-categories if $\sigma$ is a pullback square of $\bbE_1$-rings for which the functor $\Perf(B\odot^{B'}_A A') \rightarrow \Perf(B')$ is conservative (see \cite[1.3]{MR4024564} for more details). In the case of connective $\bbE_\infty$-rings, the condition appearing in \cite[1.17]{MR4024564} is satisfied if the map $\pi_0B \rightarrow \pi_0B'$ is surjective (see \cite[p.912]{MR4024564} and \cite[1.3]{MR4024564}). Consequently, \ref{excision of commutative ring spectra for perfect complexes}, which is an immediate consequence of our more general result \ref{excision of commutative ring spectra}, can be deduced from \cite[1.17]{MR4024564}. We note that our proof of \ref{excision of commutative ring spectra for perfect complexes}, which is very succinct, does not involve the construction of the $\bbE_1$-ring $B\odot^{B'}_A A'$ appearing in the statement of \cite[1.17]{MR4024564}.
\end{remark}

\begin{pg}
	Fix a finite field $\mathbb{F}_p$ of $p$ elements for some prime number $p$. Recall that an $\mathbb{F}_p$-scheme $X$ is said to be \emph{perfect} if the Frobenius map $X \rightarrow X$ is an isomorphism (see, for example, \cite[3.1]{MR3674218}). This paper was motivated by \cite[11.2]{MR3674218} of Bhargav Bhatt and Peter Scholze, which shows that the functor which carries each perfect quasi-compact and quasi-separated $\mathbb{F}_p$-scheme $X$ to the $\infty$-category of perfect complexes on $X$ (see, for example, \cite[2.8.4.4]{lurie2018sag}) is a hypercomplete sheaf with respect to the $v$-topology of \cite[3.2]{MR3674218} in the sense of \cite[p.669]{MR2522659}. We note that this result depends crucially on the fact that the derived tensor product can be identified with the ordinary tensor product for perfect rings; see \cite[3.16]{MR3674218}. 
\end{pg}

\begin{pg}
	Our second main result shows that the restriction to perfect schemes can be removed from \cite[11.2]{MR3674218} if we work in the setting of spectral algebraic geometry (see \ref{v-topology of connective ring spectra} for the notion of \emph{$v$-cover} in the spectral setting). In fact, we prove this not just for perfect modules, but for more general almost perfect modules:
\end{pg}

\begin{theorem}\label{descent of aperf}
	Let $f:A \rightarrow B$ be a $v$-cover of connective $\bbE_\infty$-rings for which the underlying map of commutative rings $\pi_0f:\pi_0A \rightarrow \pi_0B$ is of finite presentation. Then $f$ is of universal $\APerf$-descent: that is, for every morphism $A \rightarrow A'$ in the $\infty$-category $\CAlg^{\cn}$ of connective $\bbE_\infty$-rings, the induced map $\APerf(A') \rightarrow \lim \APerf({B'}^\bullet)$ is an equivalence of $\infty$-categories, where ${B'}^\bullet$ denotes the \Cech nerve of the map $A' \rightarrow A'\otimes_AB$ (formed in the opposite of the $\infty$-category $\CAlg^{\cn}$).
\end{theorem}

\begin{remark} 
	If $f:A \rightarrow B$ is a morphism of noetherian simplicial commutative rings for which the induced map $\pi_0f$ is a finitely presented $v$-cover of ordinary commutative rings, it then follows from \cite[3.3.1]{MR4560539} of Halpern-Leistner and Preygel that almost perfect complexes satisfy (not necessarily universal) descent for the morphism $f$; see \ref{descent for morphisms and functors}. Our descent result \ref{descent of aperf} can be regarded as a generalization of their work (in the affine case): it holds without the locally noetherian assumption appearing in \cite[3.3.1]{MR4560539} and is valid more generally for connective $\bbE_\infty$-rings, rather than merely for simplicial commutative rings. Moreover, the $\APerf$-descent for the morphism $f$ is universal (that is, $\APerf$ satisfies descent for arbitrary base change of the map $f$). 
\end{remark}

\begin{remark}
	Let $f:A \rightarrow B$ be a morphism of noetherian connective $\bbE_\infty$-rings for which the induced map $\pi_0f$ is a finitely presented $v$-cover of ordinary commutative rings. Then \cite[3.3.6]{MR4560539} of Halpern-Leistner and Preygel shows that the functor $\Mod^{\acn}$ satisfies (not necessarily universal) descent for the map $f$ (see \ref{properties of modules}). Using a slight variant of the proof of \ref{descent of aperf}, we can remove the noetherian assumption on $A$ and $B$ to show that $f$ is of universal $\Mod^{\acn}$-descent; see \ref{descent of almost connective modules}. 
\end{remark}

\begin{remark}\label{perf does not satisfy hyperdescent} 
	As a consequence of \ref{descent of aperf}, we will see in \ref{descent of perf} that $f$ is also of universal $\Perf$-descent. In particular, the functor $\Perf: \CAlg^{\cn} \rightarrow \Cat_\infty$ satisfies descent with respect to the Grothendieck topology on $(\CAlg^{\cn})^{\op}$ which is characterized by the following property: a sieve $\calC \subseteq (\CAlg^{\cn})^{\op}_{/A}$ is a covering if and only if it contains a finite collection of maps $\{A \rightarrow A_i \}_{1 \leq i \leq n}$ for which the induced map $f:A \rightarrow \prod A_i $ is a $v$-cover such that the underlying map of commutative rings $\pi_0f$ exhibits $\pi_0(\prod A_i )$ as a finitely presented $\pi_0A$-algebra. In contrast with \cite[11.2]{MR3674218} for ordinary perfect schemes, the sheaf $\Perf$ is not hypercomplete (that is, it is not true in general that if $U_\bullet: \Delta^{\op} \rightarrow (\CAlg^{\cn})^{\op}$ is a hypercovering with respect to the above Grothendieck topology in the sense of \cite[A.5.7.1]{lurie2018sag}, then the composition $\Delta \stackrel{U_\bullet}{\longrightarrow} \CAlg^{\cn} \stackrel{\Perf}{\longrightarrow} \Cat_\infty$ is a limit diagram). To see this, we note that for every connective $\bbE_\infty$-ring $R$, the truncation map $R \rightarrow \pi_0R$ is a $v$-cover, so that the constant cosimplicial $\bbE_\infty$-ring with value $\pi_0R$ is a hypercovering of $R$ with respect to the above Grothendieck topology. Consequently, if the functor $\Perf$ is hypercomplete, then the natural map $\Perf(R) \rightarrow \Perf(\pi_0R)$ is an equivalence of $\infty$-categories, which is false in general.
\end{remark}

\begin{remark}
	Fix integers $a \leq b$. Since the functor $\Perf$ is a sheaf for the Grothendieck topology appearing in \ref{perf does not satisfy hyperdescent} by virtue of \ref{descent of perf}, so is the subfunctor $\Perf^{[a,b]}$ which assigns to each connective $\bbE_\infty$-ring $R$ the full subcategory $\Perf^{[a,b]}(R) \subseteq \Perf(R)$ spanned by those perfect complexes whose Tor-amplitude is contained in $[a, b]$; see, for example, \cite[7.2.4.21]{lurie2017ha}. As Akhil Mathew pointed out, the functor $\Perf^{[a,b]}: \CAlg^{\cn} \rightarrow \Cat_\infty$ does not take values in the $\infty$-category of $(b-a+1)$-categories in the sense of \cite[2.3.4.1]{MR2522659}, unlike in the case of ordinary perfect schemes of \cite[11.2]{MR3674218}. Consequently, the fact that $\Perf^{[a,b]}$ is a sheaf does not guarantee that it is a hypercomplete sheaf (see \cite[6.5.2.9]{MR2522659}). In fact, for a connective $\bbE_\infty$-ring $R$, the canonical map $\Perf^{[a,b]}(R) \rightarrow \Perf^{[a,b]}(\pi_0R)$ is not an equivalence in general, so that $\Perf^{[a,b]}$ is not hypercomplete. 
\end{remark}

\begin{pg}\textbf{Conventions}.
	We will make use of the theory of $\infty$-categories and the theory of spectral algebraic geometry developed in \cite{MR2522659}, \cite{lurie2017ha}, and \cite{lurie2018sag}.
\end{pg}

\begin{pg}\textbf{Acknowledgements}. 
	The author is thankful to Bhargav Bhatt and Akhil Mathew for their helpful comments and suggestions. This work was supported by the National Research Foundation of Korea (NRF) grant funded by the Korean government (MSIT) (No. 2020R1A5A1016126) and the Sogang University Research Grant of 2023 (No. 202310024.01). 
\end{pg}

\section{The $v$-covers and universal descent morphisms of connective $\bbE_\infty$-rings}

\begin{pg}\label{v-topology for ordinary schemes}
	Let $f: X \rightarrow Y$ be a morphism of quasi-compact and quasi-separated schemes. Recall from \cite[2.1]{MR3674218} (see also \cite[2.9]{MR2679038}) that $f$ is said to be a \emph{$v$-cover} if, for every valuation ring $V$ and every morphism of schemes $\Spec V \rightarrow Y$, there exist an extension of valuation rings $V \rightarrow W$ (that is, an injective local homomorphism) and a morphism of schemes $\Spec W \rightarrow X$, which fit into a commutative diagram
$$
\xymatrix{
\Spec W \ar[r] \ar[d] & X \ar[d]^-{f}\\
\Spec V \ar[r] & Y.
}
$$
In the spectral setting, we define the \emph{$v$-cover} of connective $\bbE_\infty$-rings as follows:
\end{pg}

\begin{definition}\label{v-topology of connective ring spectra}
	Let $f: A \rightarrow B$ be a morphism of connective $\bbE_\infty$-rings (see \cite[7.1.0.1]{lurie2017ha}). We will say that $f$ is a \emph{$v$-cover} if the induced morphism $\Spec(\pi_0f): \Spec(\pi_0B) \rightarrow \Spec(\pi_0A)$ is a $v$-cover of ordinary schemes in the sense of \ref{v-topology for ordinary schemes}. 
\end{definition}

\begin{remark}
	Alternatively, we can define a \emph{$v$-cover} of affine spectral Deligne-Mumford stacks $\Spec(f): \Spec B \rightarrow \Spec A$ as in \ref{v-topology for ordinary schemes}, by virtue of the universal property of $0$-truncations of spectral Deligne-Mumford stacks (see \cite[1.4.6.3]{lurie2018sag}).
\end{remark}

\begin{pg}
	We now summarize some of the formal properties of \ref{v-topology of connective ring spectra} (which follow immediately from the case of ordinary schemes and \cite[7.2.1.23]{lurie2017ha}):
\end{pg}

\begin{lemma}\label{formal properties of v-covers}\noindent 
\begin{enumerate}[\normalfont(i)]
\item The collection of $v$-covers of connective $\bbE_\infty$-rings contains all equivalences and is stable under composition. 

\item Suppose we are given a pushout diagram of connective $\bbE_\infty$-rings
$$
\Pull{A}{B}{A'}{B'.}{f}{}{}{f'}
$$
If $f$ is a $v$-cover, then so is $f'$.
\end{enumerate}
\end{lemma}

\begin{pg}\label{universal descent morphisms}
	The notion of $v$-covers appearing in \ref{v-topology of connective ring spectra} is closely related to the notion of a \emph{universal descent morphism} of $\bbE_\infty$-rings (which makes sense more generally for $\bbE_2$-rings): recall from \cite[D.3.1.1]{lurie2018sag} that a morphism of $\bbE_\infty$-rings $f:A \rightarrow B$ is said to be a \emph{universal descent morphism} if the smallest stable subcategory of $\LMod_A$ which is closed under retracts and contains all $A$-modules of the form $N\otimes_AB$, where $N$ is a left $A$-module, coincides with $\LMod_A$. We note that for a finitely presented map $f: A\rightarrow B$ of noetherian commutative rings, the map $f$ (regarded as a morphism of discrete $\bbE_\infty$-rings) is a universal descent morphism if and only if $\Spec (f): \Spec B \rightarrow \Spec A$ is a $v$-cover of ordinary schemes; see \cite[11.26]{MR3674218}. 
\end{pg}

\begin{remark}\label{mod-descent for universal descent morphisms}
	The class of universal descent morphisms, which will play a central role in our proofs of \ref{excision of commutative ring spectra} and \ref{descent of aperf}, allows us to do descent theory in the spectral setting. In particular, if $f:A \rightarrow B$ is a universal descent morphism of $\bbE_\infty$-rings, it then follows from \cite[D.3.5.8]{lurie2018sag} that the canonical map $\Mod_A \rightarrow \lim \Mod_{B^\bullet}$ is an equivalence of symmetric monoidal $\infty$-categories, where $B^\bullet$ denotes the \Cech nerve of $f$ formed in the $\infty$-category $(\CAlg^{\cn})^{\op}$. 
\end{remark}

\begin{remark}
	Let $f: A \rightarrow B$ be a morphism of $\bbE_\infty$-rings. Let $B^\bullet$ denote the \Cech nerve of $f$ (so that $B^n$ is given by the $(n+1)$-fold tensor product $B\otimes_A \cdots \otimes_AB$ for each $n \geq 0$). We denote by $\Tot^\bullet(B/A)$ the Tot-tower of $B^\bullet$: that is, the sequence of $A$-modules 
$$
\cdots \Tot^2(B/A) \rightarrow \Tot^1(B/A) \rightarrow \Tot^0(B/A).
$$ 
Here $\Tot^n(B/A)$ denotes the limit of the diagram $\{B^m\}$ taken over the full subcategory $\Delta_{\leq n} \subseteq \Delta$ spanned by those objects $[m] \in \Delta$ such that $m \leq n$. We note that if $f: A \rightarrow B$ is a universal descent morphism, then $A$ is a retract of $\Tot^n(B/A)$ for some integer $n \geq 0$ (see the proof of \cite[D.3.2.1]{lurie2018sag}). We will use this fact to deduce some descending properties of universal descent morphisms; see \ref{descending properties of modules over universal descent morphisms}.
\end{remark}

\begin{pg}\label{properties of modules}
	Let $R$ be a connective $\bbE_\infty$-ring. Recall that an $R$-module $M$ is said to be \emph{almost connective} if it is $(-m)$-connective for some integer $m \gg 0$ (see \cite[p.1201]{lurie2017ha}). We will let $\Mod_R^{\acn} \subseteq \Mod_R$ denote the full subcategory spanned by the almost connective $R$-modules. We note that the construction $R \mapsto \Mod_R^{\acn}$ determines a functor $\Mod^{\acn}: \CAlg^{\cn} \rightarrow \Cat_\infty$. According to \cite[7.2.4.10]{lurie2017ha}, an $R$-module $M$ is \emph{almost perfect} if it is $m$-connective for some integer $m$ and is an almost compact object of the $\infty$-category $\Mod_R^{\geq m}$ of $m$-connective $R$-modules (that is, for every integer $n \geq 0$, $\tau_{\leq n}M$ is compact as an object of $\tau_{\leq n}\Mod_R^{\geq m}$). We will say that an $R$-module $M$ is \emph{perfect to order n} if, for every filtered diagram $\{N_\alpha \}$ of $0$-truncated $R$-modules, the canonical map $\colim \Ext^i_A(M, N_\alpha) \rightarrow \Ext^i_A(M, \colim N_\alpha)$ is bijective for $i<n$ and injective for $i=n$; see \cite[2.7.0.1]{lurie2018sag}. We remark that an $R$-module $M$ is almost perfect if and only if it is perfect to order $n$ for each integer $n$ (see \cite[2.7.0.2]{lurie2018sag}).
\end{pg}

\begin{pg}
	One of the main ingredients in our proofs of \ref{excision of commutative ring spectra} and \ref{descent of aperf} is the following descent result:
\end{pg}

\begin{proposition}\label{descending properties of modules over universal descent morphisms} 
	Let $f: A \rightarrow B$ be a universal descent morphism of connective $\bbE_\infty$-rings, let $M$ be an $A$-module, and let $n$ be an integer. Then:
\begin{enumerate}[\normalfont(i)]
\item The $A$-module $M$ is $n$-connective if and only if the $B$-module $M \otimes_AB$ is $n$-connective. 

\item The $A$-module $M$ is almost connective if and only if the $B$-module $M \otimes_AB$ is almost connective.

\item The $A$-module $M$ is perfect to order $n$ if and only if the $B$-module $M \otimes_AB$ is perfect to order $n$. 

\item The $A$-module $M$ is almost perfect if and only if the $B$-module $M \otimes_AB$ is almost perfect. 
\end{enumerate}
\end{proposition}

\begin{proof}
	Assertions (ii) and (iv) follow immediately from (i) and (iii), respectively. The ``only if'' directions of assertions (i) and (iii) follow from \cite[7.2.1.23]{lurie2017ha} and \cite[2.7.3.1]{lurie2018sag}, respectively. To prove ``if'' directions, we note that since $f$ is a universal descent morphism, it follows from \cite[D.3.2.1]{lurie2018sag} that $A$ is a retract of $\Tot^m(B/A)$ for some $m \geq 0$. To prove the ``if'' direction of (i), assume that $M\otimes_A B$ is $n$-connective. In particular, $M \otimes_A \Tot^m(B/A)$ is $n$-connective when viewed as a $B$-module. Since $M$ is a retract of the connective $A$-module $M \otimes_A \Tot^m(B/A)$, $M$ is $n$-connective as desired. To complete the proof of (iii), let $\{N_\alpha \}$ be a filtered diagram of $0$-truncated $A$-modules. We wish to show that the canonical map $\phi_i: \colim \Ext^i_A(M, N_\alpha) \rightarrow \Ext^i_A(M, \colim N_\alpha)$ is bijective for $i<n$ and injective for $i=n$. For this, it will suffice to show that $\phi_i$ is a retract of a bijection for $i<n$ and an injection for $i=n$. Since $A$ is a retract of $\Tot^m(B/A)$, we deduce that $N_\alpha$ is a retract of $\tau_{\leq 0}(N_\alpha \otimes_A \Tot^m(B/A))$. We are therefore reduced to proving that the horizontal map in the diagram
$$
\xymatrix{
\colim \Ext^i_A(M, \tau_{\leq 0}(N_\alpha \otimes_A \Tot^m(B/A))) \ar[r] \ar[d] & \Ext^i_A(M, \colim \tau_{\leq 0}(N_\alpha \otimes_A \Tot^m(B/A))) \ar[d] \\
\colim \Ext^i_B(M\otimes_AB, \tau_{\leq 0}(N_\alpha \otimes_A \Tot^m(B/A))) \ar[r] & \Ext^i_B(M\otimes_AB, \colim \tau_{\leq 0}(N_\alpha \otimes_A \Tot^m(B/A)))
}
$$
is bijective for $i<n$ and injective for $i=n$, where we regard $N_\alpha \otimes_A \Tot^m(B/A)$ as a $B$-module. Since the forgetful functor $\Mod_B \rightarrow \Mod_A$ commutes with the formation of $0$-truncations and preserves small colimits (see \cite[4.2.3.7]{lurie2017ha}), the vertical maps are equivalences. Consequently, the desired result follows from our assumption that $M\otimes_AB$ is perfect to order $n$ as a module over $B$. 
\end{proof}

\begin{pg}
	Before stating our next result which will be needed in our proof of \ref{descent of aperf}, we recall a bit of terminology (see \cite[3.1.1]{lz2012enhanced}):
\end{pg}

\begin{definition}\label{descent for morphisms and functors}
	Let $\calC$ be an $\infty$-category which admits fiber products, and let $\calD$ be an $\infty$-category. Let $F: \calC^{\op} \rightarrow \calD$ be a functor. We will say that a morphism $f:C_0 \rightarrow C$ in $\calC$ is \emph{of $F$-descent} if the composition $\Delta_+ \stackrel{C_\bullet}{\longrightarrow} \calC^{\op} \stackrel{F}{\longrightarrow} \calD$ is a limit diagram, where $C_\bullet$ denotes the \Cech nerve of $f$, regarded as an augmented simplicial object of $\calC$ with $C_{-1}\simeq C$ (see \cite[p.543]{MR2522659}). That is, the canonical map $F(C) \rightarrow \lim\limits_{[n]\in \Delta^{\op}} F(C_n)$ is an equivalence in $\calD$. In this case, we will also say that $F$ \emph{satisfies descent for $f$}. If this condition holds for every base change $C'\times_CC_0 \rightarrow C'$ of $f$ along a morphism $C' \rightarrow C$, we will say that $f$ is \emph{of universal $F$-descent}, or that $F$ \emph{satisfies universal descent for $f$}.
\end{definition}

\begin{theorem}\label{universal descent implies universal aperf-descent}
	Let $f:A \rightarrow B$ be a universal descent morphism of connective $\bbE_\infty$-rings. Then the map $f$ is of universal $\Mod^{\acn}$-descent (see \emph{\ref{properties of modules}}) and of universal $\APerf$-descent.
\end{theorem}

\begin{proof}
	The property of being a universal descent morphism of $\bbE_\infty$-rings is stable under pushouts (see \cite[D.3.1.6]{lurie2018sag}), so it will suffice to show that $f$ is of $\Mod^{\acn}$-descent and of $\APerf$-descent. Let $B^\bullet$ denote the \Cech nerve of $f$, so that the canonical map $\Mod_A \rightarrow \lim \Mod_{B^\bullet}$ is an equivalence of symmetric monoidal $\infty$-categories (see \ref{mod-descent for universal descent morphisms}). Using \ref{descending properties of modules over universal descent morphisms}, we deduce the desired results by restricting to the full subcategories spanned by the almost connective and almost perfect modules, respectively.
\end{proof}

\begin{corollary}\label{universal descent implies universal perf-descent}
	Let $f:A \rightarrow B$ be a universal descent morphism of connective $\bbE_\infty$-rings. Then $f$ is of universal $\Perf$-descent.
\end{corollary}

\begin{proof}
	As in the proof of \ref{universal descent implies universal aperf-descent}, it will suffice to show that $f$ is of $\Perf$-descent. Using \ref{universal descent implies universal aperf-descent}, we obtain an equivalence of symmetric monoidal $\infty$-categories $\APerf(A) \rightarrow \lim \APerf(B^\bullet)$. Then the desired result follows by restricting to the full subcategories spanned by the dualizable objects (see \cite[4.6.1.7]{lurie2017ha}), because the full subcategory of $\lim \APerf(B^\bullet)$ spanned by the dualizable objects can be identified with $\lim \Perf(B^\bullet)$ by virtue of \cite[4.6.1.11]{lurie2017ha}.
\end{proof}

\begin{remark}
	In the situation of \ref{descending properties of modules over universal descent morphisms}, it follows from \ref{universal descent implies universal perf-descent} that the $A$-module $M$ is perfect if and only if the $B$-module $M \otimes_AB$ is perfect (see \cite[3.28]{MR3459022}).
\end{remark}

\begin{remark}\label{aperf-descent implies perf-descent}
	The proof of \ref{universal descent implies universal perf-descent} shows that every morphism which is either of universal $\Mod^{\acn}$-descent or of universal $\APerf$-descent is of universal $\Perf$-descent. 
\end{remark}

\begin{remark}
	The converse of \ref{universal descent implies universal perf-descent} is not necessarily true. For example, if $R=\Sym^\ast_\mathbb{Q}(\Sigma^2 \mathbb{Q})$ denotes the free $\bbE_\infty$-algebra over $\mathbb{Q}$ on a single generator $t$ of degree $2$, then we will see from \ref{descent of perf} that the truncation map $R\rightarrow \pi_0R$ is of universal $\Perf$-descent. However, the extension of scalars functor $\Mod_R \rightarrow \Mod_{\pi_0R}$ is not conservative (since the image of the localization $R[t^{-1}]$ vanishes), so that the truncation map is not a universal descent morphism (see \ref{mod-descent for universal descent morphisms}). 
\end{remark}

\section{Proof of Theorems \ref{excision of commutative ring spectra} and \ref{descent of aperf}}

\begin{pg}
	Using the notion of a universal descent morphism of \cite[D.3.1.1]{lurie2018sag} (see also \ref{universal descent morphisms}), we now provide a proof of \ref{excision of commutative ring spectra}:
\end{pg}

\begin{proof}[Proof of \emph{\ref{excision of commutative ring spectra}}]
	We wish to show that the canonical map $\APerf(A) \rightarrow \APerf(A')\times_{\APerf(B')}\APerf(B)$ is an equivalence of $\infty$-categories. The assertion that this map is fully faithful follows immediately from the first half of \cite[16.2.0.2]{lurie2018sag}, which supplies a fully faithful functor $\Mod_A \rightarrow \Mod_{A'}\times_{\Mod_{B'}}\Mod_B$. To prove the essential surjectivity, note that we can identify the objects of $\APerf(A')\times_{\APerf(B')}\APerf(B)$ with triples $(M',N, \alpha)$, where $M'$ is an almost perfect $A'$-module, $N$ is an almost perfect $B$-module, and $\alpha: M'\otimes_{A'}B' \rightarrow N\otimes_BB'$ is an equivalence of $B'$-modules. Since almost perfect modules are almost connective (that is, $m$-connective for some integer $m$), we may assume that $M'$ and $N$ are both $n$-connective for some integer $n$. Using the second part of \cite[16.2.0.2]{lurie2018sag}, we can choose an $n$-connective $A$-module $M$ such that $M\otimes_AA' \simeq M'$ and $M\otimes_AB \simeq N$ in $\Mod_{A'}$ and $\Mod_B$, respectively. To complete the proof, it will suffice to show that $M$ is almost perfect. Since the diagram $\sigma$ is a pullback square, the map $A \rightarrow A' \times B$ is a universal descent morphism in the sense of \cite[D.3.1.1]{lurie2018sag}, so that the desired result follows from \ref{descending properties of modules over universal descent morphisms}. 
\end{proof}

\begin{pg}
	The proof of \ref{descent of aperf} will require some preliminary results.
\end{pg}

\begin{lemma}\label{almost perfect complexes and universal nilcompleteness}
	Let $f:A \rightarrow B$ be morphism of connective $\bbE_\infty$-rings. Then the canonical map $\Mod_B^{\cn} \rightarrow \lim\limits_{n \geq 0} \Mod_{B\otimes_A\tau_{\leq n}A}^{\cn}$ is an equivalence of $\infty$-categories. Moreover, it restricts to equivalences of $\infty$-categories $\APerf(B) \rightarrow \lim\limits_{n \geq 0} \APerf(B\otimes_A\tau_{\leq n}A)$ and $\Perf(B) \rightarrow \lim\limits_{n \geq 0} \Perf(B\otimes_A\tau_{\leq n}A)$.
\end{lemma}

\begin{proof}
	We first note that there is an equivalence of $\infty$-categories $\Mod_A^{\cn} \rightarrow \lim\Mod_{\tau_{\leq n}A}^{\cn}$ (see \cite[2.5.9.3]{lurie2018sag}). Then \cite[10.2.2.3]{lurie2018sag} guarantees that the natural map 
$$
\Mod_B^{\cn} \simeq \Mod_B^{\cn}\otimes_{\Mod_A^{\cn}} (\lim \Mod_{\tau_{\leq n}A}^{\cn}) \rightarrow \lim (\Mod_B^{\cn}\otimes_{\Mod_A^{\cn}}\Mod_{\tau_{\leq n}A}^{\cn})
$$ 
is an equivalence of $\infty$-categories. Since the tensor product $\Mod_B^{\cn}\otimes_{\Mod_A^{\cn}}\Mod_{\tau_{\leq n}A}^{\cn}$ can be identified with $\Mod_{B\otimes_A\tau_{\leq n}A}^{\cn}$ (see \cite[10.2.1.7]{lurie2018sag}), we can identify $\Mod_B^{\cn}$ with the limit of the diagram $\{\Mod_{B\otimes_A\tau_{\leq n}A}^{\cn}\}$ in the $\infty$-category of (not necessarily small) $\infty$-categories $\widehat{\Cat}_\infty$ of \cite[3.0.0.5]{MR2522659}. The restrictions to almost perfect and perfect modules follow immediately from \cite[2.7.3.2]{lurie2018sag}. 
\end{proof}

\begin{pg}
	Since an almost connective module is $m$-connective for some integer $m$ (see \ref{properties of modules}), we immediately obtain the following:
\end{pg}

\begin{corollary}
	Let $f:A \rightarrow B$ be morphism of connective $\bbE_\infty$-rings. Then the canonical map $\Mod_B^{\acn} \rightarrow \lim\limits_{n \geq 0} \Mod_{B\otimes_A\tau_{\leq n}A}^{\acn}$ is an equivalence of $\infty$-categories. 
\end{corollary}

\begin{lemma}\label{square-zero extensions are universal descent morphisms}
	Let $f:A \rightarrow B$ be a map of $\bbE_\infty$-rings which exhibits $A$ as a square-zero extension of $B$ by a $B$-module $M$ in the sense of \emph{\cite[7.4.1.6]{lurie2017ha}}. Then $f$ is a universal descent morphism. 
\end{lemma}

\begin{proof}
	We have a pullback diagram of $\bbE_\infty$-rings
$$
\Pull{A}{B}{B}{B\oplus \Sigma M,}{}{}{}{}
$$
where $\Sigma M$ denotes the suspension of $M$ (see \cite[7.4.1.7]{lurie2017ha}), so that the desired result follows from the definition of a universal descent morphism (see \cite[D.3.1.1]{lurie2018sag}); alternatively, the desired assertion can be deduced from \cite[11.20]{MR3674218}. 
\end{proof}

\begin{pg}
	We are now ready to give the proof of \ref{descent of aperf}. We will follow the strategy of Halpern-Leistner and Preygel in the proof of \cite[3.1.1]{MR4560539}, where square-zero extensions allow us to reduce to the case where $A$ is discrete.
\end{pg}

\begin{proof}[Proof of \emph{\ref{descent of aperf}}] 
	Since the collection of $v$-covers $f: A\rightarrow B$ between connective $\bbE_\infty$-rings for which the underlying map of commutative rings $\pi_0f:\pi_0A\rightarrow \pi_0B$ is of finite presentation is closed under pushouts (see \ref{formal properties of v-covers} and \cite[7.2.1.23]{lurie2017ha}), it will suffice to show that the map $f$ is of $\APerf$-descent (see \ref{descent for morphisms and functors}). Let $B^\bullet$ denote the \Cech nerve of $f$ as in \ref{mod-descent for universal descent morphisms}. We have a commutative diagram 
$$
\Pull{\APerf(A)}{\lim\limits_m \APerf(B^m)}{\lim\limits_n \APerf(\tau_{\leq n}A)}{\lim\limits_n \lim\limits_m \APerf(\tau_{\leq n}A\otimes_A B^m).}{}{}{}{}
$$
Using \ref{almost perfect complexes and universal nilcompleteness}, we see that the vertical maps are equivalences. Consequently, it will suffice to show that for each $n\geq 0$, the canonical map $\tau_{\leq n}A \rightarrow \tau_{\leq n}A\otimes_AB$ is of universal $\APerf$-descent. We proceed by induction on $n$. We begin with the case $n=0$. Note that the composition $\tau_{\leq 0}A \rightarrow \tau_{\leq 0}A\otimes_AB \rightarrow \pi_0(\tau_{\leq 0}A\otimes_AB)$, which can be identified with $\pi_0(f)$, is a finitely presented $v$-cover of ordinary commutative rings, so that it is a universal descent morphism by virtue of \cite[5.5]{MR4332074}. Using \ref{universal descent implies universal aperf-descent}, we deduce that it is of universal $\APerf$-descent. It then follows from \cite[3.1.2]{lz2012enhanced} that $\tau_{\leq 0}A \rightarrow \tau_{\leq 0}A\otimes_AB$ is of universal $\APerf$-descent. To carry out the inductive step, assume that the map $\tau_{\leq n-1}A \rightarrow \tau_{\leq n-1}A\otimes_AB$ is of universal $\APerf$-descent. Using \cite[3.1.2]{lz2012enhanced}, it will suffice to show that the composition $\tau_{\leq n}A \rightarrow \tau_{\leq n}A\otimes_AB \rightarrow \tau_{\leq n-1}A\otimes_AB$ is of universal $\APerf$-descent. Since the truncation map $\tau_{\leq n}A \rightarrow \tau_{\leq n-1}A$ exhibits $\tau_{\leq n}A$ as a square-zero extension of $\tau_{\leq n-1}A$ by $\Sigma^n\pi_nA$ (see \cite[7.4.1.28]{lurie2017ha}), it is of universal $\APerf$-descent by virtue of \ref{square-zero extensions are universal descent morphisms} and \ref{universal descent implies universal aperf-descent}. Combining this with the inductive hypothesis, the desired result follows from \cite[3.1.2]{lz2012enhanced} (which guarantees that the collection of universal $\APerf$-descent morphisms is closed under composition).
\end{proof}

\begin{remark}\label{descent of almost connective modules}
	Arguing as in the proof of \ref{descent of aperf} (using $\Mod^{\acn}$ in place of $\APerf$), we deduce that in the situation of \ref{descent of aperf}, $f$ is of universal $\Mod^{\acn}$-descent.
\end{remark}

\begin{corollary}\label{descent of perf}
	In the situation of \emph{\ref{descent of aperf}}, $f$ is of universal $\Perf$-descent.
\end{corollary}

\begin{proof}
	By virtue of \ref{aperf-descent implies perf-descent}, this is an immediate consequence of \ref{descent of aperf}.
\end{proof}

\begin{pg}
	Before recording some other consequences of \ref{descent of aperf}, let us introduce a bit of terminology:
\end{pg}

\begin{definition}
	Let $X:\CAlg^{\cn} \rightarrow \SSet$ be a functor, where $\SSet$ denotes the $\infty$-category of spaces (see \cite[1.2.16.1]{MR2522659}). Let $\QCoh(X)$ denote the \emph{$\infty$-category of quasi-coherent sheaves on $X$}; see \cite[6.2.2.1]{lurie2018sag}. We will say that an object $F\in \QCoh(X)$ is a \emph{perfect complex on $X$} if, for every connective $\bbE_\infty$-ring $R$ and every morphism $\eta: \Spec R \rightarrow X$, the pullback $\eta^\ast F \in \QCoh(\Spec R)\simeq \Mod_R$ belongs to the full subcategory $\Perf(R) \subseteq \Mod_R$ (see \ref{perfect complexes in the case of affines}). We let $\Perf(X)$ denote the full subcategory of $\QCoh(X)$ spanned by the perfect complexes on $X$. 
\end{definition}

\begin{remark}
	In the special case where $X$ is representable by an affine spectral Deligne-Mumford stack $\Spec A$, the full subcategory $\Perf(X) \subseteq \QCoh(X)$ corresponds to the full subcategory $\Perf(A) \subseteq \Mod_A$ under the equivalence of $\infty$-categories $\QCoh(X) \simeq \Mod_A$. 
\end{remark}

\begin{remark}
	We note that $\QCoh(X)$ admits a symmetric monoidal structure of \cite[p.497]{lurie2018sag}. According to \cite[6.2.6.2]{lurie2018sag}, an object $F\in \QCoh(X)$ is dualizable as an object of the symmetric monoidal $\infty$-category $\QCoh(X)$ if and only if it belongs to $\Perf(X)$. 
\end{remark}

\begin{pg}
	The $v$-topology on schemes is not subcanonical, but it is when restricted to quasi-compact and quasi-separated perfect schemes; see \cite[4.2]{MR3674218}. In the spectral setting, we have the following consequence of \ref{descent of aperf}:
\end{pg}

\begin{corollary}\label{qcqs spectral algebraic spaces satisfy finitely presented v-descent}
	Let $\sfX$ be a quasi-compact and quasi-separated spectral algebraic space of \emph{\cite[1.6.8.1]{lurie2018sag}}, and let $X: \CAlg^{\cn} \rightarrow \SSet$ denote the functor represented by $\sfX$ in the sense of \emph{\cite[1.6.4.1]{lurie2018sag}}. Let $f:A \rightarrow B$ be a $v$-cover of connective $\bbE_\infty$-rings for which the underlying map of commutative rings $\pi_0f:\pi_0A \rightarrow \pi_0B$ exhibits $\pi_0B$ as a finitely presented algebra over $\pi_0A$. Then $X$ satisfies universal descent for $f$.
\end{corollary}

\begin{proof}
	As in the proof of \ref{descent of aperf}, we are reduced to proving that the functor $X$ satisfies descent for the map $f$. Let $B^\bullet$ denote \Cech nerve of $f$. We have a commutative diagram
$$
\xymatrix{
\Map_{\Fun(\CAlg^{\cn}, \SSet)}(\Spec A, X) \ar[r] \ar[d] & \lim\limits_m \Map_{\Fun(\CAlg^{\cn}, \SSet)}(\Spec B^m, X) \ar[d] \\
\Fun^\otimes_{\ex}(\Perf(X), \Perf(A)) \ar[r] & \lim\limits_m \Fun^\otimes_{\ex}(\Perf(X), \Perf(B^m)),
}
$$
where $\Fun^\otimes_{\ex}(\Perf(X), \Perf(A))$ denotes the $\infty$-category of exact symmetric monoidal functors from $\Perf(X)$ to $\Perf(A)$ and $\Fun^\otimes_{\ex}(\Perf(X), \Perf(B^\bullet))$ is defined similarly (see \cite[2.1.3.7]{lurie2017ha}). We wish to show that the upper horizontal map is an equivalence. It follows from \cite[9.6.4.2]{lurie2018sag} that the vertical maps are equivalences. The desired result now follows from \ref{descent of perf}, which guarantees that the bottom horizontal map is an equivalence.
\end{proof}

\begin{pg}
	Arguing as in the proof of \ref{qcqs spectral algebraic spaces satisfy finitely presented v-descent} (using $\APerf$ and \cite[9.5.5.1]{lurie2018sag} in place of $\Perf$ and \cite[9.6.4.2]{lurie2018sag}, respectively), we obtain another consequence of \ref{descent of aperf}:
\end{pg}

\begin{corollary}
	Let $X$ be a locally noetherian geometric stack of \emph{\cite[9.3.0.1]{lurie2018sag} and \cite[9.5.1.1]{lurie2018sag}}. Let $f:A \rightarrow B$ be a $v$-cover of connective $\bbE_\infty$-rings for which the underlying map of commutative rings $\pi_0f:\pi_0A \rightarrow \pi_0B$ exhibits $\pi_0B$ as a finitely presented algebra over $\pi_0A$. Then $X$ satisfies universal descent for the morphism $f$.
\end{corollary}

\bibliography{chough_aperf}
\bibliographystyle{amsplain}

\end{document}